\documentclass[a4paper]{elsarticle}
\usepackage{graphicx} 
\usepackage{amsthm}
\usepackage{amssymb}
\usepackage{multicol}
\usepackage{tikz}
\usepackage{amsmath}
\usepackage{color}
\usepackage{geometry}
\definecolor{E}{RGB}{255,84,0} 
\definecolor{2}{RGB}{41,199,92} 
\definecolor{Amarillo}{RGB}{225,191,73}
\definecolor{Celeste}{RGB}{117,170,219}
\definecolor{a}{RGB}{132,96,255} 

\definecolor{0}{RGB}{30,123,191} 
\definecolor{1}{RGB}{255,113,102} 
\definecolor{3}{RGB}{242,207,16} 
\definecolor{5}{RGB}{255,15,154} 
\definecolor{4}{rgb}{.8,0,.8}

\pgfdeclarelayer{background}
\pgfdeclarelayer{foreground}
\pgfsetlayers{background,main,foreground}

\usepackage[english]{babel}
\usepackage[utf8x]{inputenc}
\usepackage{amsmath}
\usepackage{amsfonts}
\usepackage{enumerate}
\usepackage{comment}
\usepackage{tkz-graph}
\usetikzlibrary{decorations.pathreplacing,decorations.pathmorphing,calc}
\usepackage{wrapfig}
\usepackage{hyperref}

\newtheorem{construction}{Construction}[section]
\newtheorem{theorem}[construction]{Theorem}

\newtheorem{corollary} [construction]{Corollary}

\newtheorem{definition} [construction]{Definition}

\newtheorem{lemma}[construction] {Lemma}

\newtheorem{Claim}{Claim}

\newcommand{\adrian}[1]{\textcolor{blue}{#1}}

\date{}

\begin{document}

\begin{frontmatter}

\title{On $R$-sequenceability of odd ordered groups}
\address[AP1]{Instituto de Matem\'atica Aplicada San Luis, Universidad Nacional de San Luis and CONICET, San Luis, Argentina.}
\address[AP2]{Departamento de Matem\'atica, Universidad Nacional de San Luis, San Luis, Argentina.}

\author[AP1,AP2]{Adri\'an Pastine}
\author[AP1,AP2]{Mar\'ia Valentina Soldera Ruiz}

\begin{abstract}
We study the $R$-sequenceability of finite groups of odd order. Building on the classical theory of $R^*$-sequences and orthomorphisms, we explore two tools: the notion of $R^{**}$-sequenceability, a strengthening of $R^*$-sequenceability tailored for inductive arguments over normal subgroups with cyclic quotients, and the \textit{odd cycle index} $\tau(G)$, which measures how many orthomorphisms are required to generate a full cycle together with an involution. Our main result is a Quotient-Normal Gadget theorem, which shows that if $G$ has a normal subgroup $N$ such that $G/N$ is $R^{**}$-sequenceable and $\tau(N) \leq |G/N| - 3$, then $G$ itself is $R^{**}$-sequenceable. We prove that $\tau(G) = 2$ for cyclic groups of order coprime with $3$, and establish an inductive bound $\tau(G) \leq \max\{\tau(N), \tau(G/N)\}$ for odd ordered groups with a normal subgroup $N$. As consequences, we show that every group whose order is coprime with $30$ is $R$-sequenceable, and that every nilpotent group whose order is coprime with $6$ and not a power of $5$ is $R$-sequenceable. These results extend prior work on abelian groups to broad families of non-abelian groups.
\end{abstract}
\begin{keyword}
$R$-Sequenceability \sep
Odd-ordered Groups \sep
Nilpotent Groups \sep
Orthomorphisms \sep
Terraces \sep
Group sequencing

\textit{MSC 2020:} 20D60, 05B15, 05C25, 20F16, 20F18
\end{keyword}
\end{frontmatter}
\section{Introduction}
An $n$-ordered group $G$ is said to be \textit{$R$-sequenceable} if there exists an ordering of its non-identity elements,
\[
g_1,g_2,\ldots,g_{n-1}
\]
such that the products
\begin{align*}
    g_1^{-1}g_2,&&g_2^{-1}g_3,&&\ldots&&g_{n-2}^{-1}g_{n-1},&&g_{n-1}^{-1}g_1
\end{align*}
are all distinct. In such a case, the list
\[
g_1,g_2,\ldots,g_{n-1}
\]
is called an \textit{$R$-sequence} for $G$. As an example,  
\[
r^3\tau,r^2\tau,\tau,r^3,r^2,r\tau,r
\] 
 is an $R$-sequence for the dihedral group on $8$ elements $D_4=\langle r,\tau \mid r^4=\tau^2, r\tau=\tau r^{-1}\rangle$. Notice that its sequence of products is
\[
r,r^2,r\tau,r^3,r^3\tau,\tau,r^2\tau.
\]

The concept of $R$-sequenceability was first introduced by Ringel in \cite{ringel1974cyclic} to study embeddings of the complete graph on orientable surfaces of some genus (depending on the order of the graph). It is related to several other sequencing problems, both for groups and for combinatorial structures. For a survey on known results for sequenceability problems for groups see the article by Ollis, \cite{ollis2012sequenceable}. For a survey on results for combinatorial structures see the article by Alspach \cite{alspach2019variations}.

Since the problem of $R$-sequenceability was first introduced, it has been characterized for abelian groups across several articles (see \cite{alspach2017friedlander,friedlander1978group,headley1994r,martin2015r,ollis2011twizzler,ollis2015constructions,wang2007more}). 
\begin{theorem}[\cite{alspach2017friedlander,friedlander1978group,headley1994r,martin2015r,ollis2011twizzler,ollis2015constructions,wang2007more}]\label{th: abelianos}
Let $n\geq 3$. An abelian group of order $n$ is $R$-sequenceable if and only if its $2$-Sylow subgroup is either trivial or non-cyclic.
\end{theorem}
With respect to non-abelian groups, only some particular families  have been studied (see \cite{bedford1987,cheng2000more,keedwell1983r}). The list of known results are summarized in the following.
\begin{theorem}
    \begin{enumerate}
        \item The dihedral group of order $2n$ is $R$-sequenceable if and only if $n$ is even (\cite{cheng2000more}).
        \item The dicyclic group $Q_{4n}$ is $R$-sequenceable if and only if $n$ is even and at least $4$ (\cite{keedwell1983r}).
        \item Non-abelian groups of order $pq$ where $p$ and $q$ are prime are $R$-sequenceable (\cite{cheng2000more,keedwell1983r}).
        \item The two non-abelian groups of order 27 are $R$-sequenceable (\cite{bedford1987}).
    \end{enumerate}
\end{theorem}

The proof of Theorem \ref{th: abelianos} is divided in several cases, but the main divide is between even and odd ordered groups. In this article we  refine the construction used for odd order.
The ingredients for the proof of Theorem \ref{th: abelianos} include $R^*$-sequences and orthomorphisms. 
An \textit{$R^*$-sequence} is an $R$-sequence containing three consecutive elements, $g_1,g_2,g_3,$ such that $g_1g_3=g_2$. This concept, together with a first inductive construction, were first introduced in
\cite{friedlander1978group}, where the authors proved the following.
\begin{lemma}[\cite{friedlander1978group}]\label{cyclic} The cyclic group $Z_n$ is $R^*$-sequenceable for all
odd $n>5$.
\end{lemma}
An \textit{orthomorphism} of a group $G$ is a permutation $f(x)$, such that $g(x)=x^{-1}f(x)$ is also a permutation. Notice in particular that an $R$-sequence can be seen as an orthomorphism by considering the permutation $f:G\rightarrow G$, with $f(g_0)=g_0$, $f(g_{|G|-1})=g_1$ and $f(g_i)=g_{i+1}$ for $1\leq i \leq |G|-2$. Further notice that in $\mathbb{Z}_{2n+1}$ the permutation $f(x)=-x$ is an orthomorphism as well. We say that a product of permutations of $G$ is a \textit{$|G|$-cycle} if when looked at in disjoint cycle notation it consists of one cycle of length $|G|$. In other words, if $\pi$ is the product of the permutations and $\pi^j(x)\neq \pi^k(x)$ for every $x\in G$ and for every $1\leq j<k\leq |G|$. The other ingredient for Theorem \ref{th: abelianos} is the following inductive construction.
\begin{lemma}[\cite{alspach2017friedlander}]\label{lem gadget viejo}Let $G$ be an $R^*$-sequenceable abelian group of order $m$, let  $H$ be an odd
order abelian group and let $T_0$ be the permutation of $H$ given by $T_0(x)=-x$. If there are orthomorphisms $f_1, f_2,\ldots, f_t$ of $H$ such that
$T_0f_1f_2\ldots f_t$ is an $|H|$-cycle and $m-t-3 \geq 0$ is even, then $G\oplus H$ is $R^*$-sequenceable.
\end{lemma}
In this article we generalize the inductive construction to work instead for a (not necessarily abelian) normal subgroup and a cyclic quotient group. 
 This lets us, in particular, obtain results for nilpotent and solvable groups, thanks to their inductive nature. The main families of groups that we cover are summarized in the following.
In particular, we prove that any group whose order is coprime with $30$ is $R$-sequenceable,
and that any nilpotent group whose order is coprime with $6$ is $R$-sequenceable if in addition it is nilpotent and its order is not a power of $5$.

In Section \ref{sectionimportantres} we give an overview of our constructions, stating the results and 
the basic definitions needed to understand them. In Section \ref{sectiontautheorem}, we prove one of our main ingredients, which states that 
any group $G$ of order coprime with $6$ has an involution, $T_{(G)}$, and two more orthomorphisms, $f_1$ and $f_2$, 
such that $T_{(G)}f_1f_2$ is a $|G|$-cycle.
In Section \ref{section: graph} we present a way to visualize orthomorphisms and $R$-sequences as graphs ---a perspective that makes the gadget construction considerably more transparent---and give a first graphical introduction to our inductive construction. Then, in Section \ref{sectiong/nrdoublestar,tau(n)} we prove an inductive construction which uses
a special kind of $R$-sequence in $G/N$ together with orthomorphisms forming a cycle in $G/N$ to find
an $R$-sequence in $G$. Finally, in Section \ref{sec: conclusions} we give some closing remarks and offer 
possible future lines of work.

 \section{Preliminaries}\label{sectionimportantres}
Similar to Lemma \ref{lem gadget viejo}, our construction theorem will need two ingredients. The first is a 
particular kind of $R$-sequence, and the second is a list of orthomorphisms that form a cycle.
For the first ingredient, the conditions for $R^*$-sequence may not be enough to work inductively with $|G/N|$, 
as having $Ng_1Ng_3=Ng_3Ng_1$ may not guarantee the existence of representatives $g_1,g_3$ with
$g_1g_3=g_3g_1$. Fortunately, if the subgroup generated by $Ng_1,Ng_3$ is cyclic, there are
representatives $g_1,g_3$ such that the subgroup generated by them is cyclic, implying that $g_1$ and $g_3$ commute. As, when working with
solvable groups we can choose
 $G/N$ to be cyclic, this will be enough to obtain
our results.
This is why we introduce the following definition.
\begin{definition}\label{rdoublestar}
Let $G$ be a group. If $G$ has an $R^*$-sequence $g_1,\ldots,g_{n-1}$, with $g_i\neq e$ for $1\leq i \leq n-1$, and $g_1g_3=g_3g_1=g_2$,
such that the subgroup generated by $g_1,g_3$ is cyclic, then we say that $G$ is $R^{**}$-sequenceable, and $g_1,\ldots,g_{n-1}$ is an $R^{**}$-sequence for $G$.
\end{definition}

Notice that any $R^*$-sequenceable cyclic group is $R^{**}$-sequenceable.
Thus, Lemma \ref{cyclic} immediately yields the following.
\begin{corollary}\label{cyclicr**}
The cyclic group $Z_n$ is $R^{**}$-sequenceable for all
odd $n>5$.
\end{corollary}

For the second ingredient of our construction, we present the following definition.
\begin{definition}
Given a group $G$, the \textit{odd cycle index} of $G$, $\tau(G)$, is the smallest even number $m$ such that there exist orthomorphisms
$\theta_{1},\ldots,\theta_m$ and an involution $T_{(G)}$, such that 
$\rho=T_{(G)}\theta_{1},\ldots,\theta_m$ is a $|G|$-cycle.
\end{definition}
In Section \ref{sectiontautheorem} we prove the following theorem inductively.
\begin{theorem}\label{tautheorem}
Let $G$ be a group of odd order. Let $N$ be a normal subgroup of $G$, and $G/N$ the quotient group. Then $\tau(G)\leq \max\{\tau(N),\tau(G/N)\}$.
\end{theorem}
In \cite{alspach2017friedlander} it was shown that for any cyclic group $G$, with $\gcd(|G|,3)=1$,
$\tau(G)\leq 4$. Furthermore,  $\tau(\mathbb{Z}_3^{e})=2$ if $e\geq 2$.
We improve the first result to show that for any cyclic group, $G$, with 
$\gcd(|G|,3)=1$, $\tau(G)=2$.
These together with Theorem \ref{tautheorem} yields the following.
\begin{corollary}\label{tau=2}
Let $G$ be an odd ordered group.
If no group of the form $\mathbb{Z}_{3a}$, with $\gcd(a,3)=1$, is a normal subgroup or a quotient group of $G$, then $\tau(G)=2$. 
\end{corollary}

Having presented our ingredients, we can present our main construction theorem (which we prove in Section \ref{sectiong/nrdoublestar,tau(n)}).
\begin{theorem}[Quotient-Normal Gadget]\label{g/nrdoublestar,tau(n)}
Let $G$ be an odd ordered group and let $N$ be a normal subgroup. If $G/N$ is $R^{**}$-Sequenceable and $\tau(N)\leq |G/N|-3$, then $G$ is $R^{**}$-Sequenceable.
\end{theorem}

Applying Theorem \ref{g/nrdoublestar,tau(n)} together with Corollary \ref{cyclicr**} and Corollary \ref{tau=2} we can
find $R^{**}$-sequences for groups of order coprime with $3$ yielding the following result. Here we need
to exclude groups of order multiple of $3$, because of the $\tau$ condition (on either side of the inequality), and groups of order multiple of $5$ both because we cannot have $G/N\cong \mathbb{Z}_5$. 
\begin{corollary}
Let $G$ be an odd ordered group, with $\gcd(|G|,3)=1$. If $G$ has a normal subgroup
of prime index $p>5$, then $G$ is $R^{**}$-sequenceable.
In particular any group whose order is coprime with $30$ is $R^{**}$-sequenceable.
\end{corollary}
Recall that a nilpotent group is the direct product of its sylow subgroups. Thus, in such a case, we can choose the multiple of $5$ part of the group to be in $N$, yielding the following.
\begin{corollary}
If $G$ is a nilpotent group whose order is coprime with 6  then $G$ is $R^{**}$-sequenceable except possibly if $|G|=5^k$. 
\end{corollary}

\section{Proof of Theorem \ref{tautheorem}}\label{sectiontautheorem}
We start this section  by showing how to obtain an orthomorphism of $G$ using orthomorphisms of $N$ and of $G/N$.
This type of orthomorphisms will be used to prove Theorem \ref{tautheorem}.

In order to work with $G/N$ and $N$, we need to choose representatives of the elements of $G/N$.
Let $Ng_0,\ldots,Ng_{|G/N|-1}$ be the elements of $G/N$, and 
$g_0,\ldots,g_{|G/N|-1}\in G$ be a choice of representatives.

We first show how to build an orthomorphism from one orthomorphism of $G/N$ and several orthomorphisms of $N$.

Let $A'(Nx)$ be an orthomorphism of $G/N$ and $A(x)$ be the permutation of $\{g_0,\ldots,g_{|G/N|-1}\}$ obtained from $A'(Nx)$,
i.e.
\[
A(g_i)=g_j\text{ if and only if } A'(Ng_i)=Ng_j
\]

For $0\leq i\leq {|G/N|-1}$  let $\alpha_i$ be an orthomorphism of $N$, and let 
$\alpha=(\alpha_0,\ldots,\alpha_{|G/N|-1})$. 
For $x\in G$ write $x=ng_i$, with $n\in N$ and $g_i\in\{g_0,\ldots,g_{|G/N|-1}\}$.
We define the \textbf{$A$-lift through $\alpha$, $[A,\alpha]$} by $[A,\alpha](x)=\alpha_i(n)A(g_i)$. Notice that 
$[A,\alpha]$ may depend on the choice of representatives, $g_i\in\{g_0,\ldots,g_{|G/N|-1}\}$, which is why
we choose them before hand.
 We want to see that $[A,\alpha]$ is an orthomorphism.
This means that $x_i^{-1}[A,\alpha](x_i)\neq x_j^{-1}[A,\alpha](x_j)$ whenever $x_i\neq x_j$.
Assume then
\begin{align*}
x_i^{-1}[A,\alpha](x_i)=&x_j^{-1}[A,\alpha](x_j).\\
\end{align*}
We start by writing $x_i=n_ig_i$ and $x_j=n_jg_j$, hence 
\begin{align*}
(n_ig_i)^{-1}[A,\alpha](n_ig_i)=&(n_jg_j)^{-1}[A,\alpha](n_jg_j).\\
\end{align*}

We will first show that the equality implies $g_i=g_j$, and then that it implies $n_i=n_j$.
Notice that
\begin{align*}
(n_ig_i)^{-1}\alpha_i(n_i)A(g_i)&\in Ng_i^{-1}A(g_i)=Ng_i^{-1}A'(Ng_i),\\
(n_jg_j)^{-1}\alpha_i(n_j)A(g_j)&\in Ng_j^{-1}A(g_j)=Ng_j^{-1}A'(Ng_j).\\
\end{align*}
Thus $Ng_i^{-1}A'(Ng_i)=Ng_j^{-1}A'(Ng_j)$. But, as $A'$ is an orthomorphism, this implies $Ng_i=Ng_j$
and $g_i=g_j$.

We can now write, without loss of generality, $x_i=n_ig_1$, $x_j=n_jg_1$, and
\[
(n_ig_1)^{-1}[A,\alpha](n_ig_1)=(n_jg_1)^{-1}[A,\alpha](n_jg_1). 
\]
We want to show $n_i=n_j$. We start by developing the powers and splitting $[A,\alpha]$,
\begin{align*}
(n_ig_1)^{-1}[A,\alpha](n_ig_1)=&(n_jg_1)^{-1}[A,\alpha](n_jg_1),\\
g_1^{-1}n_i^{-1}\alpha_1(n_i)A(g_1)=&g_1^{-1}n_j^{-1}\alpha_1(n_j)A(g_1).\\ 
\end{align*}
Canceling we get $n_i^{-1}\alpha_1(n_i)=n_j^{-1}\alpha_1(n_j)$. But as $\alpha_1$ is an orthomorphism
we get $n_i=n_j$. Therefore $x_i=x_j$. We proved the following.
\begin{lemma}\label{a,alphaisortho}
$[A,\alpha](x)$ is an orthomorphism of $G$.
\end{lemma}
\begin{proof}
The result follows from the discussion preceding it.
\end{proof}
We are ready to present the proof of Theorem \ref{tautheorem}.

\begin{proof}[Proof of Theorem \ref{tautheorem}]
Notice that if $\rho=T_{(G)}\theta_1\ldots\theta_m$ is a $|G|$-cycle, then $\rho'=T_{(G)}T_{(G)}T_{(G)}\theta_1\ldots\theta_m$ 
is a $|G|$-cycle as well, since $T_{(G)}^2 = \mathrm{id}$ and inserting two copies of $T_{(G)}$ leaves the composition unchanged.

Let $m=\max\{\tau(N),\tau(G/N)\}$, $\rho_{G/N}=T^{(G/N)}_0A'^{(1)}\ldots,A'^{(m)}$ be a set of orthomorphism of $G/N$ 
generating a $|G/N|$-cycle. Let $A^{(i)}$ be the corresponding permutations of the representatives of $G/N$,
with $A^{(0)}$ corresponding to $T_{(G/N)}$.
 Notice that a list of orthomorphisms of the form 
\[
[A^{(0)},\alpha^{(0)}][A^{(1)},\alpha^{(1)}]\ldots[A^{(m)},\alpha^{(m)}]
\]
will cycle through all the different cosets,
i.e. it generates a $|G/N|$-cycle in $G/N$.
We are going to make it cycle through all the elements of $N$ the same time so that it is a $|G|$-cycle in $G$.
The choices of $\alpha^{(j)}=(\alpha_0^{(j)},\ldots,\alpha_{|G/N|-1}^{(j)})$ will have to be done carefully, as having
all the $\alpha_i^{(j)}$ be copies of the same orthomorphism may cause troubles when $|G/N|$ and $|N|$ are not coprime.
Instead we will use only one of the $\tau(N)$ orthomorphisms used to generate the cycle in $N$ for each $\alpha^{(j)}$. 

Label the elements of $G/N$ as follows,
 $Ng_0=Ne$, let $Ng_i=A_i(Ng_{i-1})$. Let $g_i$ be the chosen representative of $Ng_i$ in the group $G$.

Let $\rho_{N}=\phi^{(0)}\ldots\phi^{(m)}$ be a set of orthomorphisms of $N$ generating an $|N|$-cycle, with $\phi^{(0)}=T_{(N)}$. Let
\[
\alpha_i^{(j)}=\left\lbrace\begin{array}{lcr}
\phi^{(j)} & \text{ if } & i=j,\\
T^{(N)}_0 &  & \text{otherwise},\\
\end{array}\right. 
\]
and let $\alpha^{(j)}=(\alpha_0^{(j)},\ldots,\alpha_{|G/N|-1}^{(j)})$. 

Notice that we may have $m<|G/N|$,
which means that there will not be an $\alpha_i^{(i)}$ for every $i\in \{0,\ldots,|G/N|-1\}$. This is not an issue.

Then $\rho_{G}=[A^{(0)},\alpha^{(0)}]\ldots[A^{(m)},\alpha^{(m)}]$ generates a $|G|$-cycle.
Indeed, if we apply $\rho_G$ to an element of the form $ng_1$ we have $\rho_G(ng_1)=\rho_N(n)g_2$.
When we apply $\rho_G$ to an element of the form $ng_i$, with $i\neq 1$ we get 
$\rho_G(ng_i)=\left(T_{(N)}\right)^{m+1}(n)\rho_{G/N}(g_i)=T_{(N)}(n)\rho_{G/N}(g_i)$, 
because at no point in the composition of $\rho_G$ 
are we doing $[A^{(j)},\alpha^{(j)}](n'g_j)$ (we get $g_j$ at step $j$ of the composition if and only
if we had $g_1$ at step $1$).

Consider now the cycle that starts at $e$. The first time we apply $\rho_G$ we get $\rho_G(e)=\rho_N(e)g_2$.
When we apply it $i$ times, $1\leq i \leq |G/N|$, we get $\rho_G^i(e)=\left(T_{(N)}\right)^{i-1}(\rho_N(e))g_{i+1}$.
Thus we reach the coset $Ng_1=Ne$ again only after applying $\rho_G^{|G/N|}$. But 
$\left(T_{(N)}\right)^{|G/N|-1}(\rho_N(e))=\rho_N(e)$ because $|G/N|-1$ is even and $T_{(N)}$ is an involution.
Thus to reach $e$ again we have to apply $\left(\rho_G^{|G/N|}\right)^{|N|}$, which means that the cycle generated
has $|G|$ elements.

We have thus proved Theorem \ref{tautheorem}.
\end{proof}
\section{A slight detour to graphs}\label{section: graph}
In this section we present a way to think about orthomorphisms and $R$-sequences in graphs, and give a first graphical introduction to the construction that we formally introduce and prove in Section \ref{sectiong/nrdoublestar,tau(n)}.

It is  helpful to visualize  a permutation of a group, $f:G\rightarrow G$, as a directed graph with vertex set $G$, and with arcs of the form $(x,f(x))$ (or $x\rightarrow f(x)$). Notice that this may include loops when $x=f(x)$, and that each connected component of such a graph is a directed cycle. We think of these arcs as being colored or labelled with $x^{-1}f(x)$. When seen in this manner, a permutation is an orthomorphism if and only if in its corresponding directed graph any two arcs have different labels, see Figure \ref{Fig:Ort-Z5} for an example. Furthermore, a permutation is an $R$-sequence if and only if it is an orthomorphism and its corresponding directed graph consists of two connected components, see Figure \ref{Fig:R-sec D4} for an example. One of these components consists of a  loop and the other one of a directed cycle of length $|G|-1$.

\begin{figure}
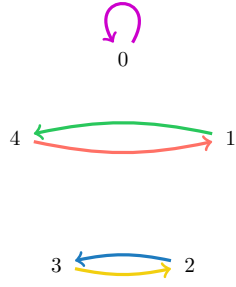

\begin{center}
\tikz{\foreach \i in {0,1,...,4}
{
\node[circle,minimum size=.05,scale=.8] at ($(90-72*\i:1.5)$) (x\i) {$\i$};
}

\draw (x0) [->,line width=1.2,color=4]
  to[out=60,in=120,looseness=8] (x0);
\draw (x2)  [->,line width=1.2,bend right=12,color=0] to (x3);
\draw (x4)  [->,line width=1.2,bend right=12,color=1] to (x1);
\draw (x1)  [->,line width=1.2,bend right=12,color=2] to (x4);
\draw (x3)  [->,line width=1.2,bend right=12,color=3] to (x2);
}
\end{center}
\caption{The orthomorphism $f(x)=-x$ on the group $\mathbb{Z}_5$.}\label{Fig:Ort-Z5}
\end{figure}

Consider a group $G$ with a normal subgroup $N$, and let
$g_0, g_1,\ldots, g_{|G/N|-1}$ be coset representatives, such that $Ng_0=N$, and that
\[
Ng_1,Ng_2,\ldots,Ng_{|G/N|-1}
\]
is an $R^{**}$-sequence for $G/N$.
Further,  assume that the permutations $T_{(N)},\theta_1,\ldots,\theta_{|G/N|-3}$ form an $|N|$-cycle for $N$. 
Let $\alpha_i=T_{(N)}$ for $0\leq i \leq 2$, and $\alpha_i=\theta_{i-2}$ for $3\leq i \leq |G/N|-1$, be permutations of $N$.
To label the vertices of the graph with the elements of the group \(G\), we proceed as follows. We arrange the elements of each coset in columns, where each row corresponds to an element of the normal subgroup \(N\). The columns are ordered according to the \(R^{**}\)-sequence of \(G/N\), beginning with \(g_1\) (the first element appearing in the \(R^*\)-sequence condition).
The arcs are determined by the orthomorphism \([A,\alpha]\), i.e., there are arcs of the form $\big(x,[A,\alpha](x)\big)$ for each $x\in G$. Notice that from columns $3$ to $|G/N|-1$ we have the orthomorphisms generating a cycle.

Next, between the columns corresponding to \(g_1\) and \(g_2\), we insert the column associated with \(g_0\), see Figure \ref{Fig:Rseq sin gadget}. Using this new column, we rearrange the previously defined arcs and add new ones generated by the elements of \(Ng_0\), excluding the identity element, see Figure \ref{Fig:Rseq con gadget}. This construction yields an involution from the column \(g_1\) to the column \(g_3\). In this way, traversing the entire graph yields the desired cycle, which is equivalent to obtaining an \(R\)-sequence of \(G\).

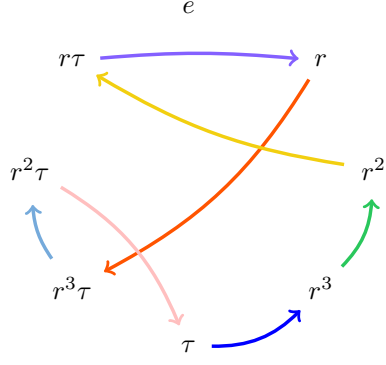
\begin{figure}
\begin{center}
\begin{tikzpicture}[line width=1, scale=0.7] 
size=.2cm,shape=circle, text=black]
\draw (6,-0.3) node [circle] (e) {$e$};%
\draw (8.5,-1.3) node [circle] (r) {$r$};%
\draw (9.5,-3.35) node [circle] (r2) {$r^2$};%
\draw (8.5,-5.6) node [circle] (r3) {$r^3$};%
\draw (6,-6.7) node [circle] (t) {$\tau$};%
\draw (3.8,-5.6) node [circle] (r3t) {$r^3\tau$};
\draw (3,-3.35) node [circle] (r2t) {$r^2\tau$};%
\draw (3.8,-1.3) node [circle] (rt) {$r\tau$};%

\draw [->,line width=1.8, color=E, bend right=-15,line width=1.3] (r) to (r3t);%
\draw [->,line width=1.8, color=Celeste, bend right=-15,line width=1.3] (r3t) to (r2t);%
\draw [->,line width=1.8, color=pink, bend right=-18,line width=1.3] (r2t) to (t);%
\draw [->,line width=1.8, color=blue, bend right=23,line width=1.3] (t) to (r3);%
\draw [->,line width=1.8, color=2, bend right=20,line width=1.3] (r3) to (r2);%
\draw [->,line width=1.8, color=3,bend right=-11,line width=1.3] (r2) to (rt);%
\draw [->,line width=1.8, color=a, bend right=-5,line width=1.3] (rt) to (r);%
\end{tikzpicture}
\end{center}
\caption{The $R$-sequence $r^3\tau,r^2\tau,\tau,r^3,r^2,r\tau,r$ with products {$\color{blue}r\tau$,
{$\color{2}r^3$,  
{$\color{3}r^3\tau$,
{$\color{a}\tau$, 
{$\color{E}r^2\tau$, 
{$\color{Celeste}r,$
{$\color{pink}r^2$
}}}}}}} for the group $D_4$.}\label{Fig:R-sec D4}
\end{figure}

\begin{figure}
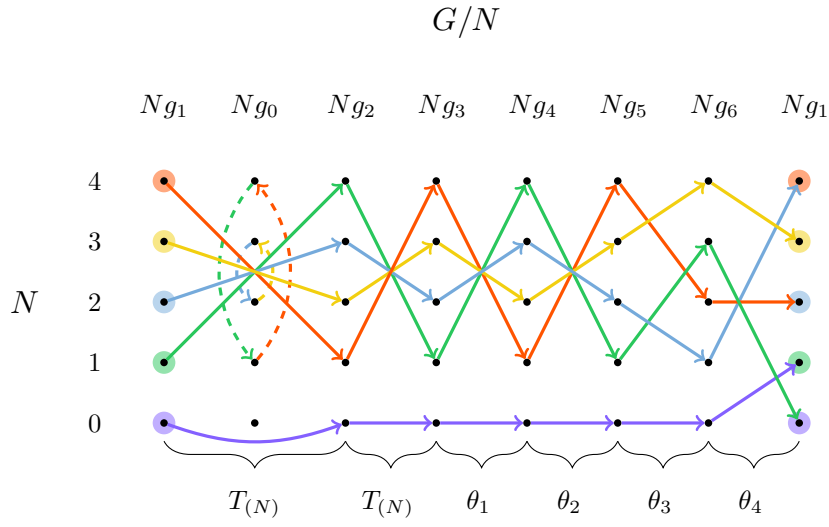

\begin{center}
\tikz[scale=0.8]{
\node[above=20pt,fill=none,draw=none,shape=rectangle,scale=1.2] at (6.75,5) {{\color{white}$G/N$}};
\node[above=20pt,fill=none,draw=none,shape=rectangle] at (0,4) {$Ng_1$};
\node[above=20pt,fill=none,draw=none,shape=rectangle] at (1.5,4) {$Ng_0$};
\node[above=20pt,fill=none,draw=none,shape=rectangle] at (3.1,4) {$Ng_2$};
\node[above=20pt,fill=none,draw=none,shape=rectangle] at (4.6,4) {$Ng_3$};
\node[above=20pt,fill=none,draw=none,shape=rectangle] at (6.1,4) {$Ng_4$};
\node[above=20pt,fill=none,draw=none,shape=rectangle] at (7.6,4) {$Ng_5$};
\node[above=20pt,fill=none,draw=none,shape=rectangle] at (9.1,4) {$Ng_6$};
\node[above=20pt,fill=none,draw=none,shape=rectangle] at (10.6,4) {$Ng_1$};
\node[above=20pt,fill=none,draw=none,shape=rectangle] at (14,3) {};
\node[right=20pt,fill=none,draw=none,shape=rectangle] at (14,3) {};
\node[left=20pt,fill=none,draw=none,shape=rectangle] at (-2,-1) {};
\node[below=20pt,fill=none,draw=none,shape=rectangle] at (-2,-1) {};

\node[above=20pt,fill=none,draw=none,shape=rectangle,scale=1.2] at (5,5.3) {$G/N$};

\foreach \i in {0,1,...,4}{
\node[left=20pt,fill=none,draw=none,shape=rectangle] at (0,\i) {$\i$};
}
\node[left=20pt,fill=none,draw=none,shape=rectangle,scale=1.2] at (-1,2) {$N$};

\node[below=15pt,fill=none,shape=rectangle,draw=none,scale=1] at (1.5,-.3) {$T_{(N)}$};
\draw [decorate,decoration={brace,amplitude=10pt},xshift=0pt,yshift=0pt,-]
(3,-.3) -- (0,-.3)node [fill=none,draw=none,black,midway,xshift=9pt]{};

\node[below=15pt,fill=none,shape=rectangle,draw=none,scale=1] at (3.7,-.3) {$T_{(N)}$};
\draw [decorate,decoration={brace,amplitude=10pt},xshift=0pt,yshift=0pt,-]
(4.5,-.3) -- (3,-.3)node [fill=none,draw=none,black,midway,xshift=9pt]{};

\node[below=15pt,fill=none,shape=rectangle,draw=none,scale=1] at (5.2,-.3) {$\theta_1$};
\draw [decorate,decoration={brace,amplitude=10pt},xshift=0pt,yshift=0pt,-]
(6,-.3) -- (4.5,-.3)node [fill=none,draw=none,black,midway,xshift=9pt]{};

\node[below=15pt,fill=none,shape=rectangle,draw=none,scale=1] at (6.7,-.3) {$\theta_2$};
\draw [decorate,decoration={brace,amplitude=10pt},xshift=0pt,yshift=0pt,-]
(7.5,-.3) -- (6,-.3)node [fill=none,draw=none,black,midway,xshift=9pt]{};

\node[below=15pt,fill=none,shape=rectangle,draw=none,scale=1] at (8.2,-.3) {$\theta_3$};
\draw [decorate,decoration={brace,amplitude=10pt},xshift=0pt,yshift=0pt,-]
(9,-.3) -- (7.5,-.3)node [fill=none,draw=none,black,midway,xshift=9pt]{};

\node[below=15pt,fill=none,shape=rectangle,draw=none,scale=1] at (9.7,-.3) {$\theta_4$};
\draw [decorate,decoration={brace,amplitude=10pt},xshift=0pt,yshift=0pt,-]
(10.5,-.3) -- (9,-.3)node [fill=none,draw=none,black,midway,xshift=9pt]{};

\foreach \i in {0,1,...,7}{
\foreach \j in {0,1,...,4}{
\node[shape=circle,minimum size=.05,fill,scale=.3] at (1.5*\i,\j) (x\i\j) {};
}}

\begin{pgfonlayer}{background}
\node[shape=circle,fill=a!50,minimum size=1cm,scale=.3] at (x00) (y00) {};
\node[shape=circle,fill=2!50,minimum size=1cm,scale=.3] at (x01) (y00) {};
\node[shape=circle,fill=Celeste!50,minimum size=1cm,scale=.3] at (x02) (y00) {};
\node[shape=circle,fill=3!50,minimum size=1cm,scale=.3] at (x03) (y00) {};
\node[shape=circle,fill=E!50,minimum size=1cm,scale=.3] at (x04) (y00) {};

\node[shape=circle,fill=a!50,minimum size=1cm,scale=.3] at (x70) (y70) {};
\node[shape=circle,fill=2!50,minimum size=1cm,scale=.3] at (x71) (y70) {};
\node[shape=circle,fill=Celeste!50,minimum size=1cm,scale=.3] at (x72) (y70) {};
\node[shape=circle,fill=3!50,minimum size=1cm,scale=.3] at (x73) (y70) {};
\node[shape=circle,fill=E!50,minimum size=1cm,scale=.3] at (x74) (y70) {};

\end{pgfonlayer}

\draw (x11) [->,dashed,color=E,line width=1.2,bend right=40] to (x14);
\draw (x12) [->,dashed,color=3,line width=1.2,bend right=60] to (x13);
\draw (x14) [->,dashed,color=2,line width=1.2,bend right=40] to (x11);
\draw (x13) [->,dashed,color=Celeste,line width=1.2,bend right=60] to (x12);

\draw (x04) [->,color=E,line width=1.2] to (x21);
\draw (x01) [->,color=2,line width=1.2] to (x24);
\draw (x02) [->,color=Celeste,line width=1.2] to (x23);
\draw (x03) [->,color=3,line width=1.2] to (x22);
\draw (x00) [->,color=a,bend right=20,line width=1.2] to (x20);
\draw (x24) [->,color=2,line width=1.2] to (x31);
\draw (x23) [->,color=Celeste,line width=1.2] to (x32);
\draw (x22) [->,color=3,line width=1.2] to (x33);
\draw (x21) [->,color=E,line width=1.2] to (x34);
\draw (x20) [->,color=a,line width=1.2] to (x30);

\draw (x30) [color=a,line width=1.2,->] to (x40);
\draw (x34) [color=E,line width=1.2,->] to (x41);
\draw (x31) [color=2,line width=1.2,->] to (x44);
\draw (x33) [color=3,line width=1.2,->] to (x42);
\draw (x32) [color=Celeste,line width=1.2,->] to (x43);

\draw (x40) [color=a,line width=1.2,->] to (x50);
\draw (x41) [color=E,line width=1.2,->] to (x54);
\draw (x44) [color=2,line width=1.2,->] to (x51);
\draw (x42) [color=3,line width=1.2,->] to (x53);
\draw (x43) [color=Celeste,line width=1.2,->] to (x52);

\draw (x50) [color=a,line width=1.2,->] to (x60);
\draw (x54) [color=E,line width=1.2,->] to (x62);
\draw (x51) [color=2,line width=1.2,->] to (x63);
\draw (x53) [color=3,line width=1.2,->] to (x64);
\draw (x52) [color=Celeste,line width=1.2,->] to (x61);

\draw (x60) [color=a,line width=1.2,->] to (x71);
\draw (x61) [color=Celeste,line width=1.2,->] to (x74);
\draw (x62) [color=E,line width=1.2,->] to (x72);
\draw (x63) [color=2,line width=1.2,->] to (x70);
\draw (x64) [color=3,line width=1.2,->] to (x73);}
\end{center}
\caption{An orthomorphism on $G$ formed through several orthomorphisms of $N$ and an $R$-sequence of $G/N$.}\label{Fig:Rseq sin gadget}
\end{figure}

\begin{figure}
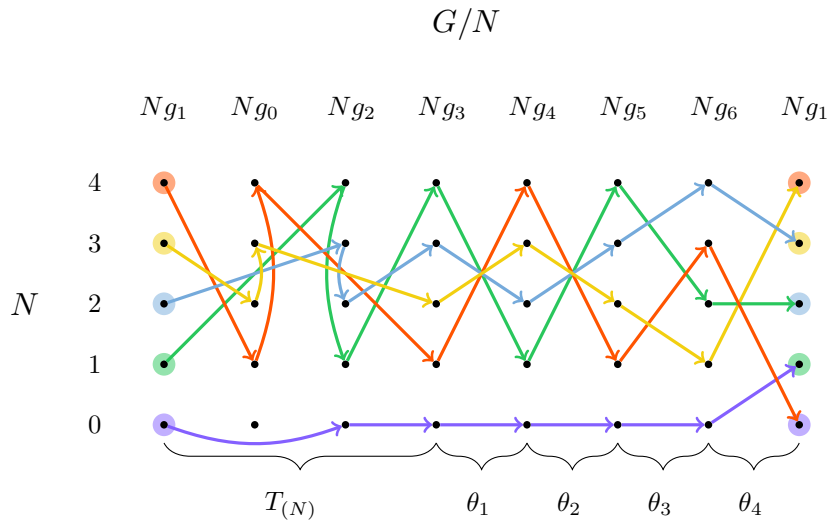

\begin{center}
\tikz[scale=0.8]{
\node[above=20pt,fill=none,draw=none,shape=rectangle,scale=1.2] at (6.75,5) {{\color{white}$G/N$}};
\node[above=20pt,fill=none,draw=none,shape=rectangle] at (0,4) {$Ng_1$};
\node[above=20pt,fill=none,draw=none,shape=rectangle] at (1.5,4) {$Ng_0$};
\node[above=20pt,fill=none,draw=none,shape=rectangle] at (3.1,4) {$Ng_2$};
\node[above=20pt,fill=none,draw=none,shape=rectangle] at (4.6,4) {$Ng_3$};
\node[above=20pt,fill=none,draw=none,shape=rectangle] at (6.1,4) {$Ng_4$};
\node[above=20pt,fill=none,draw=none,shape=rectangle] at (7.6,4) {$Ng_5$};
\node[above=20pt,fill=none,draw=none,shape=rectangle] at (9.1,4) {$Ng_6$};
\node[above=20pt,fill=none,draw=none,shape=rectangle] at (10.6,4) {$Ng_1$};
\node[above=20pt,fill=none,draw=none,shape=rectangle] at (14,3) {};
\node[right=20pt,fill=none,draw=none,shape=rectangle] at (14,3) {};
\node[left=20pt,fill=none,draw=none,shape=rectangle] at (-2,-1) {};
\node[below=20pt,fill=none,draw=none,shape=rectangle] at (-2,-1) {};

\node[above=20pt,fill=none,draw=none,shape=rectangle,scale=1.2] at (5,5.3) {$G/N$};

\foreach \i in {0,1,...,4}{
\node[left=20pt,fill=none,draw=none,shape=rectangle] at (0,\i) {$\i$};
}
\node[left=20pt,fill=none,draw=none,shape=rectangle,scale=1.2] at (-1,2) {$N$};

\node[below=15pt,fill=none,shape=rectangle,draw=none,scale=1] at (2.1,-.3) {$T_{(N)}$};
\draw [decorate,decoration={brace,amplitude=10pt},xshift=0pt,yshift=0pt,-]
(4.5,-.3) -- (0,-.3)node [fill=none,draw=none,black,midway,xshift=9pt]{};

\node[below=15pt,fill=none,shape=rectangle,draw=none,scale=1] at (5.2,-.3) {$\theta_1$};
\draw [decorate,decoration={brace,amplitude=10pt},xshift=0pt,yshift=0pt,-]
(6,-.3) -- (4.5,-.3)node [fill=none,draw=none,black,midway,xshift=9pt]{};

\node[below=15pt,fill=none,shape=rectangle,draw=none,scale=1] at (6.7,-.3) {$\theta_2$};
\draw [decorate,decoration={brace,amplitude=10pt},xshift=0pt,yshift=0pt,-]
(7.5,-.3) -- (6,-.3)node [fill=none,draw=none,black,midway,xshift=9pt]{};

\node[below=15pt,fill=none,shape=rectangle,draw=none,scale=1] at (8.2,-.3) {$\theta_3$};
\draw [decorate,decoration={brace,amplitude=10pt},xshift=0pt,yshift=0pt,-]
(9,-.3) -- (7.5,-.3)node [fill=none,draw=none,black,midway,xshift=9pt]{};

\node[below=15pt,fill=none,shape=rectangle,draw=none,scale=1] at (9.7,-.3) {$\theta_4$};
\draw [decorate,decoration={brace,amplitude=10pt},xshift=0pt,yshift=0pt,-]
(10.5,-.3) -- (9,-.3)node [fill=none,draw=none,black,midway,xshift=9pt]{};

\foreach \i in {0,1,...,7}{
\foreach \j in {0,1,...,4}{
\node[shape=circle,minimum size=.05,fill,scale=.3] at (1.5*\i,\j) (x\i\j) {};
}}

\begin{pgfonlayer}{background}
\node[shape=circle,fill=a!50,minimum size=1cm,scale=.3] at (x00) (y00) {};
\node[shape=circle,fill=2!50,minimum size=1cm,scale=.3] at (x01) (y00) {};
\node[shape=circle,fill=Celeste!50,minimum size=1cm,scale=.3] at (x02) (y00) {};
\node[shape=circle,fill=3!50,minimum size=1cm,scale=.3] at (x03) (y00) {};
\node[shape=circle,fill=E!50,minimum size=1cm,scale=.3] at (x04) (y00) {};

\node[shape=circle,fill=a!50,minimum size=1cm,scale=.3] at (x70) (y70) {};
\node[shape=circle,fill=2!50,minimum size=1cm,scale=.3] at (x71) (y70) {};
\node[shape=circle,fill=Celeste!50,minimum size=1cm,scale=.3] at (x72) (y70) {};
\node[shape=circle,fill=3!50,minimum size=1cm,scale=.3] at (x73) (y70) {};
\node[shape=circle,fill=E!50,minimum size=1cm,scale=.3] at (x74) (y70) {};
\end{pgfonlayer}

\draw (x00) [color=a,line width=1.2,->,bend right=20] to (x20);
\draw (x20) [color=a,line width=1.2,->] to (x30);

\draw (x01) [color=2,line width=1.2,->] to (x24);
\draw (x24) [color=2,line width=1.2,->,bend right=20] to (x21);
\draw (x21) [color=2,line width=1.2,->] to (x34);

\draw (x04) [color=E,line width=1.2,->] to (x11);
\draw (x11) [color=E,line width=1.2,->,bend right=20] to (x14);
\draw (x14) [color=E,line width=1.2,->] to (x31);

\draw (x02) [color=Celeste,line width=1.2,->] to (x23);
\draw (x23) [color=Celeste,line width=1.2,->,bend right=20] to (x22);
\draw (x22) [color=Celeste,line width=1.2,->] to (x33);

\draw (x03) [color=3,line width=1.2,->] to (x12);
\draw (x12) [color=3,line width=1.2,->,bend right=20] to (x13);
\draw (x13) [color=3,line width=1.2,->] to (x32);

\draw (x30) [color=a,line width=1.2,->] to (x40);
\draw (x34) [color=2,line width=1.2,->] to (x41);
\draw (x31) [color=E,line width=1.2,->] to (x44);
\draw (x33) [color=Celeste,line width=1.2,->] to (x42);
\draw (x32) [color=3,line width=1.2,->] to (x43);

\draw (x40) [color=a,line width=1.2,->] to (x50);
\draw (x41) [color=2,line width=1.2,->] to (x54);
\draw (x44) [color=E,line width=1.2,->] to (x51);
\draw (x42) [color=Celeste,line width=1.2,->] to (x53);
\draw (x43) [color=3,line width=1.2,->] to (x52);

\draw (x50) [color=a,line width=1.2,->] to (x60);
\draw (x54) [color=2,line width=1.2,->] to (x62);
\draw (x51) [color=E,line width=1.2,->] to (x63);
\draw (x53) [color=Celeste,line width=1.2,->] to (x64);
\draw (x52) [color=3,line width=1.2,->] to (x61);

\draw (x60) [color=a,line width=1.2,->] to (x71);
\draw (x61) [color=3,line width=1.2,->] to (x74);
\draw (x62) [color=2,line width=1.2,->] to (x72);
\draw (x63) [color=E,line width=1.2,->] to (x70);
\draw (x64) [color=Celeste,line width=1.2,->] to (x73);}

\end{center}
\caption{An $R$-sequence formed after changing some arcs in Figure~\ref{Fig:Rseq sin gadget}.}\label{Fig:Rseq con gadget}
\end{figure}

\section{Proof of Theorem \ref{g/nrdoublestar,tau(n)}}\label{sectiong/nrdoublestar,tau(n)}

In this section we present the proof of Theorem \ref{g/nrdoublestar,tau(n)}. Let $Ng_1,\ldots,Ng_{|G/N|-1}$
be an $R^{**}$-sequence of $G/N$. Let $A'(Ng_i)=Ng_{i+1}$ and $A'(Ne)=Ne$. Let $A$ be the corresponding
permutation of the chosen representatives $g_0,\ldots,g_{|G/N|-1}$, with $Ng_0=Ne$  and $Ng_1Ng_3=Ng_3Ng_1=Ng_2$. 

Let $\tau(N)\leq |G/N|-3$, and let $T_{(N)},\theta_{1},\ldots,\theta_{|G/N|-3}$ be orthomorphisms of $N$
generating an $|N|$-cycle (obtainable because $\tau(N)-|G/N|$ is even, and we can add an even number of $T_{(N)}$ ).
Let $\rho=T_{(N)}\theta_{1}\ldots\theta_{|G/N|-3}$.
Let $\alpha_i=\theta^{(i-3)}$ for $3\leq i \leq |G/N|-1$ and $\alpha_i=T_{(N)}$ for $0\leq i \leq 2$.
Then $[A,\alpha]$ is an orthomorphism by Lemma \ref{a,alphaisortho}.
Notice that $[A,\alpha]$ acts as $\rho$ from $g_3$ to $g_{|G/N|-1}$. 

We first show that the construction uses the same generators for its arcs as $[A,\alpha]$,
\textbf{(Step 1)} proving that it is an orthomorphism; afterwards we will show that it generates a $(|G|-1)$-cycle,
\textbf{(Step 2)} proving that it is an $R$-sequence (here is where we need the $R^*$-condition); and \textbf{(Step 3)} that it satisfy the $R^{**}$-condition (here is where we need the $R^{**}$-condition).

We will need to look at a permutation of a group $f$ as the set of arcs $(g,f(g))$ in the complete Cayley
digraph $\overrightarrow{K}(G)$.

Let $E$ be the set of arcs between elements of $Ng_0,Ng_1,Ng_2,$ and $Ng_3$, i.e.,
\[
E=\left\lbrace g\in G \mid [A,\alpha](x)=xg \text{ with }x\in\{Ng_0,Ng_1,Ng_2\}\right\rbrace.
\]
As $[A,\alpha]$ is an orthomorphism, it is easy to see that
\[
E=N\cup N g_1^{-1}g_2\cup N g_2^{-1}g_3.
\]
Focus on the subgraph induced by $Ng_1$, $Ng_2$ and  $Ng_3$ (i.e., keep those vertices and the arcs between them). Notice that this subgraph is formed by vertex-disjoint paths of length 2.

As $T_{(N)}$ is an involution, we can 
partition $N$ into pairs of the form $\{n,T_{(N)}(n)\}$. Later we will give the partition more precisely, but for now it is enough to know that it is a partition.
Now, for each pair $n,T_{(N)} (n)$, $n\neq e$, replace the directed $2$-paths
\begin{align*}
\left(ng_1,T_{(N)}(n)g_2,ng_3\right),&&\left(T_{(N)} (n)g_1,ng_2,T_{(N)}(n)g_3\right),
\end{align*}
and the directed cycle
\[
\left(ne,T_{(N)}(n)e,ne\right)
\]
by the directed $3$-paths 
\[
\left(ng_1,T_{(N)} (n)e,ne,T_{(N)} (n)g_3\right)\ \text{and}\ \left(T_{(N)} (n)g_1,ng_2,T_{(N)} (n)g_2,ng_3\right).
\]

The directed 2-path $(g_1,g_2,g_3)$ is left unaltered. The new collection of directed paths is called the \textit{gadget} on columns $g_1,e,g_2,g_3$. Notice that each vertex of the form $mg_1$ is connected through a path in the gadget to $T_{(N)}(m)g_3$. In that sense, we say that the gadget acts on $N$ as $T_{(N)}$.
Let $\Phi:\{g_0,g_1,g_2\}\rightarrow \{g_1,g_2,g_3\}$ be the function obtained from this change, and let $\Psi:G\rightarrow G$ be the permutation defined by
\[
\Psi(ng_i)=\begin{cases}
\Phi(ng_i)&\text{if $i\in\{0,1,2\}$}\\
[A,\alpha](ng_i)&\text{otherwise.}
\end{cases}
\]
Our aim is to show that $\Phi$ induces and $R^{**}$-sequence. The first step is to show that $(ng_i)^{-1}\Phi(ng_i)\in E$ for $i\in \{0,1,2\}$.

We are going to prove the following.
\begin{lemma}
The arcs of the gadget use the same colors as the arcs of $[A,\alpha]$ on columns $g_1,g_0,g_2,g_3$. Furthermore,each path of the gadget acts on $N$  as $T_{(N)}$.
\end{lemma}
\begin{proof}
The result follows from the discussion preceding it.
\end{proof}
Arcs of the form $\left(T_{(N)} (n)e,ne\right)$ and arcs of the form $\left(ng_2,T_{(N)} (n)g_2\right)$ are called \textit{vertical} arcs. Arcs that are not vertical are called \textit{diagonal} arcs. 

Notice that the vertical arc $\left(T_{(N)}(n)e,ne\right)$ is generated by the group element $T_{(N)} (n)^{-1}n$, and the arc $\left(ng_2,T_{(N)} (n)g_2\right)$ is generated by the group element $n^{-1}T_{(N)}(n)$. As $n,T_{(N)}(n)$ run through all possible pairs, these use each element of $N$ as a generator for an arc exactly once.

Lets look at the diagonal arcs of the gadget now. Here we will actually make a choice on which element is $n$ and which element is $T_{(N)}(n)$ in our pairs. We want to find a subset $R\subset N$ such that $T_{(N)}(R)\cap R=\emptyset$ and $T_{(N)}(R)\cup R=N$. The way we want to set up is that, if in the gadget we use
\[
\left(ng_1,T_{(N)}(n)e,ne,T_{(N)}(n)g_3\right)\ \text{and}\ \left(T_{(N)}(n)g_1,ng_2,T_{(N)}(n)g_2,ng_3\right),
\]  
then $n\in R$. If we manage to find such an $R$, notice that $n\in R$ if and only if $T_{(N)}(n) \notin R$. Also if $n\in R$, then $C(n)=\left\lbrace  m\in G| \exists g\in G,\ g^{-1}m^{-1}T_{(N)}(m)g=n^{-1}T_{(N)} (n)\right\rbrace\in R$.  We need to show that it is possible to choose $R$ such that we use each element as an arc generator exactly once. We will do this by defining a group action on $N$ and taking orbits, in such a way that $C(n)$ and $T_{(N)}$ are disjoint orbits of the action. In order to properly define the action and prove that it satisfies what we need, 
the following property will be useful.  
\begin{Claim}\label{gng}
If $g$ and $n$ have odd order, then $g^ing^{-i}\neq n^{-1}$ for any $i$.
\end{Claim}

\begin{proof}
Assume $g^ing^{-i}= n^{-1}$. As $n$ has odd order, $n^{-1}= n^{2k}$ for some $k$. Then 
\[
g^in^{-1}g^{-i}= g^in^{2k}g^{-i}=n^{-2k}=n.
\]  
Let $2j+1$ be the order of $g$. 

On the one hand,
\begin{align*}
(g^i)^{2j+1}n(g^{-i})^{2j+1}&=(g^i)^{2j}n^{-1}(g^{-i})^{2j}\\
&=(g^i)^{2j-1}n(g^{-i})^{2j-1}\\
&=(g^i)n(g^{-i})\\
&=n^{-1}.
\end{align*}
On the other hand, 
\[
(g^i)^{2j+1}n(g^{-i})^{2j+1}=ene=n.
\]
A contradiction! Therefore $g^ing^{-i}\neq n^{-1}$.
\end{proof}

As $T_{(N)}$ is an orthomorphism, the map $\psi:G\rightarrow G$ defined by $\psi(g)=g^{-1}T_{(N)}$ is a permutation. It is straightforward to see that $\rho_g=\psi^{-1}\left(g\psi(n)g^{-1}\right)$ defines a group action from $G$ on $G$. Let $C(n)$ be the orbit of $n$ under this action. We want to show that   $T_{(N)}(n)\not\in C(n)$ and that $T_{(N)}(C(n))=C(T_{(N)}(n))$. If $T_{(N)}(n)\in C(n)$, letting $x=n^{-1}T_{(N)}(n)$, we would have 
$x^{-1}=gxg^{-1}$, for some $g\in G$. This would contradict Claim \ref{gng}. Furthermore,  $m\in C(n)$ if and only if then there is some $g\in G$ such that
\[
g^{-1}m^{-1}T_{(N)}(m)g=n^{-1}T_{(N)}(n).
\]
Taking inverses on both sides, this implies
\begin{align*}
g^{-1}T_{(N)}(m)mg=&T_{(N)}(n)^{-1}n\\
g^{-1}T_{(N)}(m)T_{(N)}(T_{(N)}(m))g=&T_{(N)}(n)^{-1}T_{(N)}(T_{(N)}(n)).
\end{align*}
This means that $T_{(N)}(m)\in C(T_{(N)}(n))$. By a cardinality argument, this implies 
\[T_{(N)}(C(n))=C(T_{(N)}(n)).\] 
Thus, $D(n)=C(n)\cup C(T_{(N)}(n))$ defines an equivalence class for the action.

Let $n_1, \ldots, n_k$ be such that
\begin{equation*}
G=\bigcup_{i=1}^k D(n_i).
\end{equation*}
and such that 
\begin{equation*}
D(n_i)\neq D(n_j)\ \text{if}\ i\neq j
\end{equation*}
This implies $T_{(N)}(n_i)\neq n_j$ for each $1< i,j \leq k$. Let
\begin{equation*}
R=\bigcup_{i=1}^kC(n_i).
\end{equation*}
Notice that $T_{(N)}(R)\cap R=\emptyset$ and that $R\cup T_{(N)}(R)=G$.

Consider now the arcs in the gadget of the forms $(ne,T_{(N)}(n)g_3)$ and $(T_{(N)} (n)g_1, ng_2)$. The generators are of the forms $n^{-1}T_{(N)} (n)g_3$ and $g^{-1}_1T_{(N)}(n)^{-1}ng_2$. If $g_1g_3=g_2$, then $g_3=g_1^{-1}g_2$ and the generators are in the same coset. We want to be careful and make sure that they are distinct for any $n_1,n_2 \in R$. Assume that $n_1^{-1}T_{(N)} (n_1)g_1^{-1}g_2=g_1^{-1}T_{(N)} (n_2)^{-1}n_2g_2$.
\begin{align*}
n_1^{-1}T_{(N)} (n_1)g_1^{-1}g_2&=g_1^{-1}T_{(N)} (n_2)^{-1}n_2g_2\\
n_1^{-1}T_{(N)}(n_1)g_1^{-1}&=g_1^{-1}T_{(N)} (n_2)^{-1}n_2\\
n_1^{-1}T_{(N)}(n_1)&=g_1^{-1}T_{(N)}(n_2)^{-1}n_2g_1.
\end{align*}
Which means that, $T_{(N)}(n_2)\in C(n_1)\subseteq R$. But, as  $R\cap T_{(N)}(R)=\emptyset$, this is a contradiction.

The arcs that we have left to consider are of the forms $(ng_1, T_{(N)}(n)e)$ and $(T_{(N)}(n)g_2, ng_3)$. They are generated by arcs of the forms $g_1^{-1}n^{-1}T_{(N)}(n)$ and arcs of the forms $g_2^{-1}T_{(N)}(n)^{-1}ng_3$, respectively. If $g_3g_1=g_2$, then $g_1^{-1}=g_2^{-1}g_3$, the generators are in the same coset, and we want to make sure that they are distinct. Assume that $g_2^{-1}g_3n_1^{-1}T_{(N)}(n_1)=g_2^{-1}T_{(N)}(n_2)^{-1}n_2g_3$ 
\begin{align*}
g_2^{-1}g_3n_1^{-1}T_{(N)}(n_1)&=g_2^{-1}T_{(N)}(n_2)^{-1}n_2g_3\\
g_3n_1^{-1}T_{(N)}(n_1)&=T_{(N)}(n_2)^{-1}n_2g_3\\
n_1^{-1}T_{(N)}(n_1)&=g_3^{-1}T_{(N)}(n_2)^{-1}n_2g_3,
\end{align*}    
Which means that, $T_N(n_2)\in C(n_1)\subseteq R$, contradicting $R\cap T_{(N)}(R)=\emptyset$.

Further notice that after changing the arcs in columns $g_0$, $g_1$, $g_2$ and $g_3$ for the gadget, the resulting permutation acts on $N$ as $T_{(N)}$ from $g_0$ to $g_3$ and as $\theta^{i-3}$ from $g_i$ to $g_{i+1}$. Hence, it acts on $N$  as $\rho$. Thus,  starting at $eg_1$ we go through each element of $Ng_1$, which in turn implies that we go through all of $G\setminus{e}$.  Furthermore, the fact that the gadget uses the same arcs as $[A,\alpha]$ implies that it is a permutation. Thus it is an $R$-sequence. Finally, the directed $2$-path $(g_1,g_2,g_3)$ gives the $R^{**}$-condition. This completes the proof of Theorem \ref{g/nrdoublestar,tau(n)}.

\section{Conclusions}\label{sec: conclusions}
In this manuscript we presented a construction for $R$-sequences on odd ordered groups based on $R$-sequences and the existence of certain orthomorphisms. Combining this with known constructions we were able to show that groups of order coprime with $30$, and nilpotent groups of order coprime with $6$ and divisible by at least one prime other than $5$ are $R$-sequenceable. The obstacles in order to include other groups with our construction are as follows.
\begin{itemize}
    \item \textbf{Finding $\tau(G)$ for non-abelian $3$-groups.} Currently, it is known that $\tau(A)=2$ for every abelian $3$-group other than $\mathbb{Z}_3$. Finding this is the key to include groups that are not coprime with $3$ in the general result.
    \item \textbf{Finding $R^{**}$-sequences for non-abelian $5$-groups.} The current problem with groups of order a power of $5$ is that $\mathbb{Z}_5$ cannot be used as the quotient group $G/N$ for the construction, as it is not $R^{*}$-sequenceable. This, in part, excludes groups of order $5^k$. In order to include them, and to include certain solvable groups with particular quotient groups, an $R^{**}$-sequence for groups of order $5^k$ would be crucial.
\end{itemize}
On a separate note, the construction may be altered to attack solvable groups of even order. For this to work, one would need to use an $R$-sequence for the ``even'' part, and a set of permutations generating a cycle for the ``odd'' part. Here the issue is that the current construction asks for the $R$-sequence on the quotient group. For some groups a new construction changing the roles of $N$ and $G/N$ would be needed. There is hope that this new construction could work for non-solvable groups, assuming that their non-simple parts are $R$-sequenceable.
\bibliographystyle{acm}
\bibliography{bibliography.bib}

\end{document}